\documentclass[10pt]{article}

\usepackage{amsthm,amsmath,amsfonts,amssymb,bm,bbm,mathtools} 

\date{October 2025}
\newif\ifdraft

\ifdraft
\date{DRAFT: \today}
\fi

\iffalse 

  \usepackage[landscape,includehead,headheight=12pt,headsep=10pt,left=\columnsep,right=\columnsep,top=12.98999pt,bottom=\columnsep,]{geometry}
  \twocolumn
\else
  \ifdraft\usepackage[notref,notcite]{showkeys}\fi
\fi

\usepackage{pdfsync,hyperref}

\newtheorem{thm}{Theorem}
\newtheorem{prop}{Proposition}
\newtheorem{lem}{Lemma}
\newtheorem{cor}{Corollary}

\theoremstyle{definition}
\newtheorem*{definition}{Definition}
\newtheorem{example}{Example}

\theoremstyle{remark}
\newtheorem{remark}{Remark}

\newcommand{\D}{\partial}
\newcommand{\inv}{^{-1}}
\newcommand{\etal}{\textit{et al.}}

\newcommand{\supr}{^{(r)}}

\renewcommand{\epsilon}{\varepsilon}

\newcommand{\R}{\mathbb{R}}
\newcommand{\C}{\mathbb{C}}

\newcommand{\N}{\mathbb{N}}
\newcommand{\PP}{\mathbb{P}}
\newcommand{\HH}{\mathbb{H}}

\newcommand{\nwc}{\newcommand}
\nwc{\qref}[1]{(\ref{#1})}
\nwc{\ip}[1]{\langle #1 \rangle}
\nwc{\np}{{n+1}}
\nwc{\one}{\mathbbm{1}}

\nwc{\supk}{^{(k)}}
\nwc{\supj}{^{(j)}}

\makeatletter
\@mparswitchfalse
\makeatother

\makeatletter
\newcommand{\eqlab}[1]{\leavevmode\hfill\refstepcounter{equation}\label{#1}\textup{\tagform@{\theequation}}}
\makeatother

\newcommand{\warning}[1]{\typeout{}\typeout{WARNING: #1 at line \the\inputlineno}\typeout{}}
\newenvironment{todo}[1][TODO]{%
    \ifdraft\else\warning{TODO still present in final version}\fi
    \MakeFramed{\advance\hsize-\width \FrameRestore}\textbf{#1. }}%
    {\endMakeFramed}
    {\end{todo}}

\DeclareMathOperator{\im}{{\rm Im}}
\DeclareMathOperator{\re}{{\rm Re}}
\DeclareMathOperator{\Log}{{\rm Log}}
\DeclareMathOperator{\esssup}{{\rm ess\,sup}}

\title{On generating functions of Hausdorff moment sequences}
       
\author{Jian-Guo {\sc Liu} 
\thanks{Departments of Physics and Mathematics,
Duke University, Durham, NC 27708, USA. Email:jliu@phy.duke.edu}
\and Robert L. {\sc Pego}
\thanks{Department of Mathematical Sciences
and Center for Nonlinear Analysis,
Carnegie Mellon University, Pittsburgh, PA 15213, USA.
        Email: rpego@cmu.edu }
       }

\begin{document}
\maketitle

\begin{abstract}
The class of generating functions for completely monotone sequences
(moments of finite positive measures on $[0,1]$)
has an elegant characterization as the class of Pick functions analytic and positive on $(-\infty,1)$.
We establish this and another such characterization and develop a variety of consequences.
In particular, we characterize generating functions for 
moments of convex and concave probability distribution functions on $[0,1]$.
Also we provide a simple analytic proof that for any real $p$ and $r$ with $p>0$, 
the Fuss-Catalan or Raney numbers $\frac{r}{pn+r}\binom{pn+r}{n}$, $n=0,1,\ldots$ 
are the moments of a probability distribution on some interval $[0,\tau]$
{if and only if} $p\ge1$ and $p\ge r\ge 0$. The same statement holds for the binomial coefficients
$\binom{pn+r-1}n$, $n=0,1,\ldots$.

A corrigendum ({\it Trans.\ Amer.\ Math.\ Soc.}, to appear) has been included as an appendix, 
correcting gaps in the proof of Lemma~3.
\end{abstract}




\section{Introduction}

Given a finite positive (Borel) measure $\mu$ on $[0,1]$, its sequence of moments 
\begin{equation}\label{e:cjcm}
c_j = \int_0^1 t^j\,d\mu(t)\,, \quad j=0,1,\dots,
\end{equation}
is {\em completely monotone}. This means that 
\[
(I-S)^k c_j \ge0 \quad\mbox{ for all $j,k\ge0$,}
\]
where $S$ denotes the backshift operator given by $Sc_j = c_{j+1}$ for $j\ge0$, so that
\[
(I-S)^k c_j = \sum_{n=0}^k (-1)^n 
\binom{k}{n} c_{n+j} \, .
\]
A classical theorem of Hausdorff  states that complete 
monotonicity characterizes
such moment sequences: A sequence $c=(c_j)_{j\ge0}$ is 
the moment sequence of a finite positive measure on $[0,1]$
if and only if it is completely monotone \cite[p.~115]{Widder}.

It is well recognized \cite{Roth2008,Ruscheweyh2009}
that the generating function of a completely monotone sequence,
\begin{equation}\label{e:F0def}
F(z) = \sum_{j=0}^\infty c_j z^j = \int_0^1 \frac{1}{1-zt}\,d\mu(t) \,,
\end{equation}
belongs to the set $P(-\infty,1)$, consisting of all \textit{Pick functions} 
analytic on $(-\infty,1)$.  Recall what this means \cite{Donoghue}:
A function $f$ is a {Pick function} if $f$ is analytic in the 
upper half plane $\HH=\{z\in \C:\im z>0\}$ and leaves it invariant,
satisfying
\begin{equation}\label{e:Pickprop}
(\im z)\im f(z)\ge0
\end{equation}
 for all $z$ in the domain of $f$.
If $(a,b)$ is an open interval in the real line,
$P(a,b)$ denotes the set of Pick functions that are analytic on
$(a,b)$. This means they take real values on $(a,b)$
and admit an analytic continuation by reflection from the upper half plane 
across the interval $(a,b)$.

Conditions on $F$ which are equivalent to complete monotonicity of $(c_j)$
have been provided by  Ruscheweyh \etal~\cite{Ruscheweyh2009} and are discussed below.
Yet the following simple and useful characterizations, though closely related to known
results  (see especially Theorems~2.8-2.11 in \cite{BendatSherman1955}),
appear to have escaped explicit attention in the long literature on the subject.

\begin{thm}\label{t:F}
Let $c=(c_j)_{j\ge0}$ be a real sequence with generating function $F$
and upshifted generating function 
\begin{equation}\label{d:FB}
F_1(z) = z F(z) = \sum_{j=0}^\infty  c_j z^{j+1}\,.
\end{equation}
Then the following are equivalent.
\begin{itemize}
\item[(i)] 
$c$ is completely monotone.
\item[(ii)] $F$ is a Pick function that is analytic and nonnegative on $(-\infty,1)$.
\item[(iii)] $F_1$
is a Pick function that is analytic on $(-\infty,1)$, with $F_1(0)=0$.
\end{itemize}
\end{thm}
To explain the terminology,
we remark that $F_1$ is the generating function of the {\em upshifted} sequence
$S^\dag c$ given by $S^\dag c_j = c_{j-1}$ for $j\ge1$ and $S^\dag c_0= 0$.
The Pick-function conditions on $F$ and $F_1$ in (ii) and (iii) mean that $F$
and $F_1$ belong to $P(-\infty,1)$, so their domain contains $\C\setminus[1,\infty)$.
Whether it is more convenient to verify the Pick property for $F_1$,
or the Pick property and the positivity condition for $F$,
seems likely to depend on the context or the application. 


It is convenient here to  summarize explicit criteria for $(c_j)_{j\ge0}$ 
to be the moment sequence for a probability distribution supported on 
an interval of the form $[0,\tau]$. This is equivalent to saying the generating
function satisfies
\begin{equation}\label{e:Fab}
F(z) = \int_0^\tau \frac1{1-tz}\,d\mu(t) \,.
\end{equation}

\begin{cor}\label{c:ab} Let $\tau>0$, and let
$c=(c_j)_{j\ge0}$ be a real sequence with $c_0=1$ and
generating function $F$.
Then the following are equivalent.
\begin{itemize}
\item[(i)] 
$c$  is the moment sequence  of
a probability measure $\mu$ on $[0,\tau]$.
\item[(ii)] $F$ is a Pick function analytic and 
nonnegative on $(-\infty,1/\tau)$ with $F(0)=1$.
\item[(iii)] $F_1(z)=zF(z)$ is a Pick function analytic 
on $(-\infty,1/\tau)$ with $F_1'(0)=1$.
\item[(iv)] The sequence $(c_j\tau^{-j})_{j\ge0}$ is completely monotone.
\end{itemize}
\end{cor}

We shall omit the simple proof by dilation  from Theorem \ref{t:F}.
Due to \qref{e:Fab}, we will call sequences that satisfy the conditions
of Corollary~\ref{c:ab} \textit{dilated Hausdorff moment sequences}.
We note that when $t<0$, $1/(1-tz)$ is not a Pick function.
This has the consequence that the moment generating function $F$
for a probability distribution with nontrivial support in $(-\infty,0)$ 
is not a Pick function.

The proof of Theorem~\ref{t:F} is provided in section 2 below. 
As shown there, the equivalence of (i) and (ii)  may be inferred from 
the proof of Lemma~2.1 of \cite{Ruscheweyh2009}.
In section 3 we will describe a number of direct consequences 
of Theorem~\ref{t:F}, including criteria for complete monotonicity
of probability distributions obtained by randomization from 
exchangable trials, and a simple proof of
infinite divisibility of completely monotone probability
distributions on $\N_0=\{0,1,\ldots\}$.
In section~4,  building on work of 
Diaconis and Freedman~\cite{DiaconisFreedman2004}
and Gnedin and Pitman~\cite{GnedinPitman2008},
we characterize generating functions of moments of distribution
functions on $[0,1]$ that are either convex or concave.  

Finally, in section~5 we use Corollary~\ref{c:ab}  to provide
simple analytic proofs that for any real $p$ and $r$ with $p>0$,
the Fuss-Catalan numbers (or Raney numbers)
\[
\frac{r}{pn+r}\binom{pn+r}{n}\ , \quad n=0,1,\ldots  
\]
form a dilated Hausdorff moment sequence if and only if 
$p\ge1$ and $p\ge r\ge0$, and the same holds for the binomial sequence
\[
\binom{pn+r-1}n\ , \quad n=0,1,\ldots  .
\]
These results were proved previously
by Mlotkowski~\cite{Mlotkowski} and
Mlotkowski and Penson~\cite{MlotkowskiP14}
using free probability theory and
monotonic convolution arguments.
In section 5 we also use the results of section 4 to prove
 the existence of nonincreasing ``canonical densities'' 
associated with Fuss-Catalan sequences. 
The moments of corresponding canonical distributions turn out
to be a binomial sequence with $r=1$. Moreover, an explicit
formula for the inverse of the canonical distribution can be derived
from the following integral representation
formula for binomial coefficients: For any integer $k>0$ and real $r\ge k$,
\begin{equation}\label{e:rk}
\binom{r}k = \frac1\pi \int_0^\pi \left(\frac{ \sin x}{\sin^\theta\theta x \,\sin^{1-\theta}(1-\theta)x}\right)^r \,dx\, ,
\qquad \theta = \frac kr\ . 
\end{equation}
This formula follows from a statement in the proof of Proposition 2 in \cite{Simon2014}.

\section{Characterization}

The purpose of this section is to prove Theorem~\ref{t:F}.
First we establish the equivalence of (i) and (iii).
Supposing $F_1\in P(-\infty,1)$ as stated in (iii), our goal is 
to prove that \qref{e:cjcm} holds. 
We observe that since $F_1$ is a Pick function, 
so also is the function defined by 
\begin{equation}\label{d:Fstar}
F_*(z) = -F_1(1/z).
\end{equation}
This function is analytic on the cut plane $\C\setminus [0,1]$, with $F_*(z)\to 0$
as $|z|\to\infty$. The main representation theorems for Pick functions 
\cite[pp.~20 and 24]{Donoghue} imply  there are 
real numbers $\alpha_*\ge0$ and $\beta_*$,  
and a locally finite measure $\mu$ on $\R$ with 
increasing distribution function (denoted the same as is conventional) 
$\mu\colon\R\to\R$,
such that
\begin{equation}\label{e:Grep}
F_*(z) = \alpha_* z + \beta_* + \int_\R \left( \frac{1}{t-z}-\frac t{t^2+1}\right) d\mu(t).
\end{equation}
Moreover,  $\mu$ is characterized by the limits\footnote{See the Corrigendum below for a correction.}
\begin{equation}\label{e:mudef}
\mu(b)-\mu(a) =\mu(a,b]= \lim_{h\to0^+} \frac1\pi\int_a^b \im F_*(t+ih) \,dt, \qquad a,b\in\R.
\end{equation}
Because $\im F_*(t)=0$  for $t\in\R\setminus[0,1]$,
 the measure $\mu$ has support only on the cut $[0,1]$, 
 and it is a finite measure with total mass $\mu\R=\mu[0,1]$.
The fact that $F_*(z)\to 0$ as $z\to-\infty$ forces $\alpha_*=0$ and the cancellation of constant
terms, whence
\begin{equation}\label{e:Fsrep}
F_*(z) =  \int_0^1 \frac1{t-z}\,d\mu(t) .
\end{equation}
For $|z|<1$, then, 
geometric series expansion and use of Fubini's theorem yields
\begin{equation}\label{e:cm5}
F_1(z)=-F_*(1/z)= \int_0^1\frac{z}{1-zt}\,d\mu(t) = \sum_{j=0}^\infty z^{j+1}\int_0^1t^j\,d\mu(t).
\end{equation}
Then \qref{e:cjcm} follows immediately upon comparison with the definition of $F_1$.

For the converse, given $(c_{j})_{j\ge0}$  completely monotonic,
let $\mu$ satisfy \qref{e:cjcm} and define $F_*$ by \qref{d:Fstar}.
Then \qref{e:cm5} holds for $|z|<1$, and \qref{e:Fsrep} follows for $|z|>1$.
But this shows that $F_*$ is Pick and analytic on $\C\setminus[0,1]$, whence $F_1$ is
Pick and analytic on $\C\setminus[1,\infty)$, which means $F_1\in P(-\infty,1)$.
Hence (iii) and (i) are equivalent.

Next we show that (i) and (ii) are equivalent. We know that (i) implies (ii) due to
\qref{e:F0def}.  Conversely, suppose (ii), and 
define 
$\hat F(z)=-F(1/z).$
Then $\hat F$ is a Pick function analytic on $\C\setminus[0,1]$, thus
has a representation analogous to \qref{e:Grep} as
\begin{equation}
\hat F(z) = \hat \alpha z + \hat \beta + 
\int_\R \left( \frac{1}{t-z}-\frac t{t^2+1}\right) d\hat\mu(t).
\end{equation}
As before, $\hat\mu$ is a finite measure with support in $[0,1]$, and the fact
that $\hat F(z)\to-F(0)$ as $z\to-\infty$ forces $\hat\alpha=0$ and
\begin{equation}\label{e:hatFe1}
\hat F(z) = -F(0) + \int_0^1 \frac1{t-z}\,d\hat\mu(t).
\end{equation}
Since $F$ is a Pick function analytic and nonnegative on $(-\infty,1)$, necessarily
$F'(z)\ge0$ there, and $\lambda=\lim_{z\to-\infty}F(z)$ exists and is nonnegative.
Therefore taking $z\to0$ from the left in \qref{e:hatFe1} implies  
\[
\lambda = -F(0) + \int_0^1 \frac{d\hat\mu(t)}t\ .
\]
In particular this is finite, hence
\begin{equation}
\hat F(z) = \lambda + \int_0^1 \left(\frac{1}{t-z}-\frac1t \right)\,d\hat\mu = 
\int_0^1 \frac{z}{t-z}\,d\mu(t),
\end{equation}
where $\mu = \lambda\delta_0+\hat\mu/t$ is a finite measure on $[0,1]$.
By consequence, $F(z)=-\hat F(1/z)$ has an integral representation as in 
\qref{e:F0def}, and (i) follows.  This finishes the proof of Theorem~\ref{t:F}.

\begin{remark} \label{r:mass}
The mass that the measure $\mu$ assigns to the left end of its support is
\begin{equation}
\mu\{0\} = \lambda= \lim_{x\to-\infty} F(x)  = \lim_{x\to-\infty} \frac{F_1(x)}{x} \,.
\end{equation}
As is classically known~\cite[p.\ 164]{Widder}, one has
$c_{j}=f(j)$ for some
completely monotonic $f$ defined on $[0,\infty)$ if and only if the term
$c_0$ is minimal, and this is the case if and only if the measure
$\mu$ has no atom at $0$. 

The mass that the measure $\mu$ assigns to the right end of its
support is determined in a similar way.  In the scaled expression \qref{e:Fab}, one has
\begin{equation}\label{e:rtend}
\mu\{\tau\} = \lim_{x\uparrow 1/\tau} (1-\tau x)F(x)\,.
\end{equation}
\end{remark}


\begin{remark} Our proof of the equivalence of (i) and (ii) of Theorem~\ref{t:F}
basically follows arguments from  \cite{Ruscheweyh2009}
in a slightly different order. To compare, note that
after a trivial normalization to make $c_0=1$, 
Lemma 2.1 of \cite{Ruscheweyh2009} shows that
 $(c_j)_{j\ge0}$ is completely monotone if and only if
$F\in P(-\infty,1)$ and satisfies the additional conditions
\begin{itemize}
\item[(a)] {$\lim_{n\to\infty} F(z_n)/z_n=0$ for some sequence
$z_n\in\C$ with $\im z_n\to+\infty$, and $\im z_n\ge \delta\re z_n$
for some positive constant $\delta$,}
\item[(b)] $\limsup_{x\to-\infty} F(x)\ge0.$
\end{itemize}
If $F\in P(-\infty,1)$, then due to the Pick property \qref{e:Pickprop}, $F$ is 
increasing in the real interval $(-\infty,1)$. Consequently, condition (b)
is equivalent to nonnegativity of $F$ on $(-\infty,1)$.
Moreover, the additional condition (a) is superfluous, by the lemma below. 
Thus equivalence of (i) and (ii) follows from the proof of Lemma 2.1 in \cite{Ruscheweyh2009}.
\end{remark}

%

\begin{lem}  Suppose $F\in P(-\infty,1)$ and (b) holds. Then (a) holds.
\end{lem}
\begin{proof}
As noted in \cite{Ruscheweyh2009}, $F$ has a representation analogous
to \qref{e:Grep}, of the form
\begin{equation}\label{eq:Fpick}
F(z) = a + b z + \int_1^\infty \frac{1+tz}{t-z}\,d\sigma(t) \ ,
\end{equation}
where $a\in\R$, $b\ge0$, and  $\sigma$ is a finite measure on $[1,\infty)$.
Note that as $z\to-\infty$ along the real axis, we have
\begin{equation}\label{e:Picklim}
\frac1z
\int_1^\infty \frac {1+tz}{t-z}\,d\sigma(t) \to 0
\end{equation}
by the dominated convergence theorem. Consequently,
if $b>0$ then $F(z)\to-\infty$ as $z\to-\infty$. Thus, if (b) holds then $b=0$.
Then it follows that (a) holds with $z_n=in$, since with $z=z_n$, 
\qref{e:Picklim} holds again.
\end{proof}


\section{Consequences}
We next develop a few straightforward consequences of Theorem~\ref{t:F}. 
Recall that a composition of Pick functions is a Pick function.  
Using this and related facts, we get
ways of constructing and transforming completely monotone sequences 
through Taylor expansion of compositions.  

\subsection{Dilation}
Suppose $(c_j)_{j\ge0}$ is completely monotone and $p>0$.
Let
\[
G(z)= F(pz+1-p)= \sum_{j=0}^\infty c_{j} \sum_{k=0}^j \binom{j}{k} (1-p)^{j-k}p^k z^k.
\]
Then by Theorem~\ref{t:F}, $G$ is a Pick function analytic and nonnegative
on $(-\infty,1)$, as a consequence of the same properties for $F$.
Provided $0<p<2$ so that $|1-p|<1$, then for $|z|$ small we can write
\[
G(z) = \sum_{k=0}^\infty b_{k} z^k, \qquad 
b_{k} = \sum_{j=k}^\infty c_{j} \binom{j}{k}  (1-p)^{j-k}p^k.
\]
Then it follows $(b_k)_{k\ge0}$ is completely monotone. 

This application has a familiar interpretation in terms of compound 
probability distributions \cite[\S XII.1]{Feller1}:
Let the random variable
$S_N$ be the number of successes in $N$ independent Bernoulli trials, 
each  successful with probability $p\in(0,1)$.  
Suppose $N$ is itself random, independent of the Bernoulli trials, with 
\begin{equation}\label{e:Ndist}
\PP\{N=j\}=c_{j}\ ,\quad j\ge0.
\end{equation}
  Then the compound probability of having exactly $k$ successes is
\[
\PP\{S_N=k\} = \sum_{j=k}^\infty \PP\{N=j\}\PP\{S_j=k\} = b_{k}\,, \quad k\ge0,
\]
with $b_{k}$ as above. Thus, if $(c_j)$ is completely monotone, then so is $(b_k)$.

By averaging over $p$, we can also treat the case when independent trials are replaced by exchangeable ones,
due to a well-known theorem of B.\ de Finetti. Let $S_n$ denote
the number of successes in the first $n$ trials of an infinite exchangeable sequence
of successes and failures. Then by de Finetti's theorem  \cite[VII.4]{Feller},
there is a probability distribution $\nu$ on $[0,1]$
such that 
\begin{equation}
\PP\{S_n=k\} = \binom{n}{k}\int_0^1 t^k(1-t)^{n-k}\,d\nu(t).
\end{equation}
If $N$ is selected randomly as before, the probability of $k$ successes becomes
\begin{equation}\label{e:deFin}
\hat b_{k}:= \PP\{S_N=k\} = 
\sum_{j=k}^\infty c_{j} \binom{j}{k}\int_0^1 (1-t)^{j-k}t^k\,d\nu(t) .
\end{equation}
By approximating $\nu$ using discrete measures, and using the fact that 
convex combinations and pointwise
limits of completely monotone sequences are completely monotone, it follows that
$(\hat b_k)_{k\ge0}$ is completely monotone if $(c_j)_{j\ge0}$ is.

We remark that one can show that complete monotonicity of $(\hat b_k)_{k\ge0}$ persists
(though with loss of the probabilistic interpretation)
if the averaging in \qref{e:deFin} over $[0,1]$ is replaced by one over
a compact interval $[0,\tau]\subset[0,2)$:
\begin{equation}\label{e:deFin2}
\hat b_{k} = 
\sum_{j=k}^\infty c_{j} \binom{j}{k}\int_0^\tau (1-t)^{j-k}t^k\,d\nu(t) .
\end{equation}

\subsection{Reflection}
Let $U$ be any Pick function analytic on $(0,\infty)$. 
If $U$ has the additional property that the right limit $U(0^+)$ exists,
then it is a {\em complete Bernstein function}. (The theory of these functions,
including an extensive table of examples, is developed in \cite{Schilling_etal_Bernstein}.
Our Theorem 1 may be regarded as a discrete analog of Theorem 6.2 
in \cite{Schilling_etal_Bernstein}, which is the main
characterization theorem for complete Bernstein functions.)
Now, fix $r>0$ and let 
\begin{equation}\label{d:F1U}
F_1(z)=U(r)-U(r(1-z))\ .
\end{equation}
Then for $|z|<1$ we may write
\begin{equation}\label{e:exU1}
F_1(z)= \sum_{n=1}^\infty c_{n-1}z^n,
\qquad c_{n-1}= (-1)^{n-1} U^{(n)}(r) \frac{r^n}{n!}, \quad n\ge1.
\end{equation} 
Then $F_1$ is in $P(-\infty,1)$, hence the sequence $(c_n)_{n\ge0}$ is completely monotone
by Theorem~\ref{t:F}.
Conversely, if $(c_n)_{n\ge0}$ is completely monotone 
and $U(rz) = -F_1(1-z)$, then $U$ is a Pick function analytic on $(0,\infty)$,
and $U$ is completely Bernstein if  the left limit $F_1(1^-)<\infty$.

\subsection{Composition I}
Next, consider more general types of compositions. Let $c^{*j}_k$ denote 
the $k$-th term of the standard $j$-fold 
convolution of the sequence $c=(c_n)_{n\ge0}$
with itself, which we can write in terms of the Kronecker delta function $\delta_k^n$ 
and $\N_0=\N\cup\{0\}$ as
\[
c^{*j}_k = \sum_{n\in\N_0^j} \delta_k^{n_1+\ldots+n_j} c_{n_1}\dots c_{n_j}\,.
\]
Recall that if $b=(b_j)_{j\ge0}$ is a sequence with generating function $G$, 
then $A=G\circ F$ is the generating function for the sequence $a=(a_k)_{k\ge0}$
defined by 
\begin{equation}\label{e:compound}
a_k = \sum_{j=0}^\infty b_j c^{*j}_k\ , \quad k\ge0.
\end{equation}
By Theorem \ref{t:F}, $a$ is completely monotone 
if and only if $G\circ F$ is a  Pick function analytic and nonnegative on $(-\infty,1)$,
or equivalently $z G\circ F(z)$ is a Pick function analytic on $(-\infty,1)$.

If $b$ and $c$ are probability distributions on $\N_0$, 
then $a$ is a compound $b$ distribution with compounding distribution $c$.
In this case we have the following.

\begin{prop} 
Suppose  that $b$ and $c$ are probability distributions on $\N_0$
and that $c$ and $S b=(b_{j+1})_{j\ge0}$ are completely monotone, with $b_0\ge0$. 
Then $a$, given by \qref{e:compound},
is completely monotone.
\end{prop}

\begin{proof}
By Theorem \ref{t:F}, $F$ is  a Pick function
analytic and nonnegative on $(-\infty,1)$ with $F(1)=1$, 
and $G$ is a Pick function analytic and nonnegative on $[0,1)$ with $G(1)=1$. 
Then $G\circ F$ is a Pick function analytic and nonnegative on $(-\infty,1)$,
hence $a$ is completely monotone.
\end{proof}

\subsection{Infinite divisibility}

The concept of infinite divisibility has been called ``the core of the now classical
limit theorems of probability theory --- the reservoir created by the contributions
of innumerably many individual streams and developments'' \cite[Chap. XVII]{Feller}.
Recall, for example,  that a probability distribution $\mu$
on $\R$ is infinitely divisible if and only if it can arise
as the weak limit of the law of the sum of $n$ independent random variables
 with common distribution $\nu_n$ (see \cite[Thm.~XVII.1.1]{Feller} 
 or \cite[Thm.~5.2]{SteutelvanHarn}).
This means that the 
distribution function $\mu$ is the pointwise limit of $n$-fold convolutions:
 At each point of continuity of $\mu$,
\begin{equation}
\mu(x) = \lim_{n\to\infty} \nu_n^{*n}(x) \ .
\end{equation}
Theorem 1 provides an alternative route to the following long-known result
for distributions supported on $\N_0$
(see \cite{Steutel} and Theorem~10.4 of \cite{SteutelvanHarn}).

\begin{prop}
If $(c_j)$ is a probability distribution on $\N_0$
and $(c_j)$ is completely monotone, then it is infinitely divisible.
\end{prop}
\begin{proof}
For $0<r<1$ the function $G(z)=z^r$ is a Pick function, positive for $z>0$. 
If $c=(c_j)_{j\ge0}$ is completely monotone, then $F_r(z)=F(z)^r$ is a Pick function
analytic and nonnegative on $(-\infty,1)$, hence $F_r$
is the generating function for a completely monotone sequence $(a_k)_{k\ge0}$
by Theorem~\ref{t:F}.
For $r=1/n$ we have $a^{*n}_k=c_k$, because $(F_{1/n})^n=F$.
It follows $(c_j)$ is infinitely divisible.
\end{proof}

\subsection{Convolution groups and canonical sequences}
It is interesting to note that for any nontrivial completely monotone $c$ and any
real $r$, there is a simple algorithm for computing the sequence 
$a\supr$ with generating function $F(z)^r$,
arising from the characterization theorems of Hansen and Steutel~\cite{HansenSteutel}
(see \cite[Thm.~10.5]{SteutelvanHarn}).  Suppose $c_0=1$
for convenience. First compute the \textit{canonical sequence} $(b_k)_{k\ge0}$, so that
\begin{equation}
(n+1)c_{n+1} = \sum_{k=0}^n c_{n-k}b_k\,,\quad n=0,1,\ldots,
\end{equation}
then determine $a\supr$ by setting $a\supr_0=c_0^r=1$ and requiring
\begin{equation}
(n+1) a_{n+1}\supr = r \sum_{k=0}^n a_{n-k}\supr b_k\,, \quad n=0,1,\ldots.
\end{equation}
The generating functions $F$, $G$, and $A_r$, for $c$, $b$, and $a\supr$ respectively,
 satisfy
\begin{equation}
F'(z)= F(z)G(z), \quad F(0)=1\,,
\end{equation}
\begin{equation}
A'_r(z) = r A_r(z)G(z), \quad A_r(0)=1\ ,
\end{equation}
whence it follows $A_r(z)=F(z)^r$, with
\begin{equation}\label{d:tG1}
F(z)= \exp \tilde G_1(z), \qquad
\tilde G_1(z)=\int_0^z G(t)\,dt =
 \sum_{n=0}^\infty \frac{b_n z^{n+1}}{n+1}\ .
\end{equation}
The resulting set  $\{a\supr:r\in\R\}$ 
forms a convolution group: $a^{(r+s)}=a\supr * a^{(s)}$.

By \cite[Thm.~4]{HansenSteutel},
the sequence $\tilde b=(b_n/(n+1))_{n\ge0}$ is completely monotone,  
and has an interesting characterization in terms of moments
of contractive distribution functions: Namely, 
$\tilde G_1$ is the upshifted generating function for $\tilde b$, and
there is a \textit{canonical density} $w:[0,1]\to[0,1]$ (measurable)
such that 
\begin{equation}\label{e:logF}
\Log F(z) = \tilde G_1(z)
=  \int_0^1 \frac z{1-tz}\,w(t)\,dt\,.
\end{equation}
(The distribution function $\nu(t)=\int_0^t w(s)\,ds$ 
is contractive: $|\nu(t+s)-\nu(t)|\le |s|$.)

We note a curious formula for $(b_n)$.
By Taylor expansion of $\Log$,
\[
 \tilde G_1(z)= \sum_{j=1}^\infty \frac{(-1)^{j-1}}{j} (F(z)-1)^j,
\]
and  the coefficient of $z^n$ in $(F(z)-1)^j$
vanishes for $j>n$. Binomial expansion, $F(z)^k=A_k(z)$,
and use of a well-known identity \cite[(5.10)]{Graham} yields 
\begin{equation}\label{e:bnformula}
\frac{b_{n-1}}{n} =
 \sum_{j=1}^n
\sum_{k=1}^j \frac{(-1)^{k-1}}{j} \binom{j}{k} a^{(k)}_n
= \sum_{k=1}^n \frac{(-1)^{k-1}}{k} \binom{n}{k} a^{(k)}_n.
 \end{equation}

The algorithm above extends without change to the 
case of any dilated Hausdorff moment sequence $(c_j)$, 
i.e., a sequence satisfying the conditions of 
Corollary~\ref{c:ab}.
Since $F^r=(F^s)^{r/s}$ is a Pick function whenever
$F^s$ is Pick and $0\le r\le s$, and
pointwise limits of Pick functions are Pick,
$a\supr$ is guaranteed to be a dilated Hausdorff moment sequence 
for $r$ in some \textit{maximal interval}, typically of the form $[0,r_*]$. 
We may characterize this interval in terms of 
a canonical density as follows. 
Dilation of  \qref{e:logF} shows that the sequence $c=a^{(1)}$
has generating function $F=A_1$ given by a canonical density 
$w:[0,\tau]\to[0,1]$ according to
\begin{equation}\label{e:logA1}
\Log A_1(z) =  \int_0^\tau \frac z{1-tz}\,w(t)\,dt\,.
\end{equation}

\begin{prop}\label{p:rho}
Let $c=(c_n)_{n\ge0}$ be a dilated Hausdorff sequence, satisfying
the conditions of Corollary~\ref{c:ab}. 
Let $F$ be its generating function, and let $A_r=F^r$ be the 
generating function of $a\supr$ as above. 
Then $a\supr$ is a dilated Hausdorff sequence if and only if
$\ 0\le r \esssup w(t)\le 1$.
\end{prop}
\begin{proof}
By the Corollary to Lemma~5 in \cite{HansenSteutel}, one has the following.
Recall $\HH=\{z\in \C\colon\im z>0\}$.
\begin{lem}
$\arg A_1(z)\in(0,\pi\rho)$ for all $z\in\HH$ if and only if 
$|w(t)|\le \rho$ for a.e.~$t$. 
\end{lem}
\noindent
By consequence, we find that since $\arg A_1(z)=0$ for $z<1$,
the image
\begin{equation}\label{e:argA1}
\arg A_1(\HH) = (0,\pi\rho_*), \quad \rho_*={\rm ess\,sup\,}|w(t)|.
\end{equation}
It follows that $A_r=A_1^r$ is a Pick function  if and only if $0\le r\rho_*\le1$.
Then the result follows from Corollary~\ref{c:ab}.
\end{proof}

\subsection{Composition II}
If, instead of composing generating functions directly, we
compose their upshifted versions, we obtain the following.

\begin{prop}\label{p:compose}
Let $b$ and $c$ be 
real sequences with upshifted generating functions $G_1$ and $F_1$
respectively. Suppose $G_1\circ F_1$ is a Pick function analytic on $(-\infty,1)$. 
Then the sequence $\hat a$ defined by 
\[
\hat a_{k-1} = 
\sum_{j=1}^k b_{j-1} c^{*j}_{k-j}\ , \quad k\ge1,
\]
has upshifted generating function $\hat A_1=G\circ F_1$, and
is completely monotone.
\end{prop}

\begin{proof}
Compute $G_1\circ F_1(z)=\sum_{k=1}^\infty \hat a_{k-1}z^k$ via
\[
G_1\circ F_1(z) = \sum_{j=1}^\infty b_{j-1} z^j\left(\sum_{n=0}^\infty c_n z^n\right)^j =
\sum_{j=1}^\infty b_{j-1} z^j \sum_{n\in\N_0^j} c_{n_1}\cdots c_{n_j}  z^{n_1+\ldots+n_j}
\]
collecting terms to obtain the coefficient of $z^{k}$, $k\ge1$,
which we label as $a_{k-1}$.  
\end{proof}

\begin{example}\label{ex:reflect}
  Let $c=(c_j)_{j\ge0}$ be any completely monotone sequence. 
Then the sequence $(\hat c_k)_{k\ge0}$ of its {\em leading differences}, given by
\begin{equation}\label{d:leading}
\hat c_k = (I-S)^kc_0 = \sum_{j=0}^k \binom{k}{j}  (-1)^jc_j \ , \quad k\ge0,
\end{equation}
is completely monotone.  Actually, this is easiest to establish directly from the Hausdorff
representation \qref{e:cjcm}, since binomial expansion yields
\[
\sum_{j=0}^k \binom{k}{j}  (-1)^j\int_0^1 t^j\,d\mu(t) = \int_0^1 (1-t)^k \,d\mu(t) = 
\int_0^1 t^k \,d\hat \mu(t),
\]
where $\hat \mu(t)= \mu(1)-\mu(1-t)$ is obtained by reflection.  
It is a charming fact that taking leading differences of leading differences
gives back the original sequence. (For a more general inversion formula see
\cite[VII.1]{Feller}.)
Here, though, for later use in section~\ref{s:convex} we wish to point out how this 
is related to Theorem \ref{t:F} and Proposition~\ref{p:compose}: 
The sequence $(\hat c_k)_{k\ge0}$ in \qref{d:leading} has
upshifted generating function given by 
\begin{equation}\label{e:hatF1}
\hat F_1(z) = -F_1(H(z)), \qquad H(z)= \frac{-z}{1-z}=1-\frac1{1-z}\ .
\end{equation}
This is true because
\[
(1-z)^{-j-1} = \sum_{n=0}^\infty 
\binom{n+j}{n} z^n 
 \ ,\]
hence
\begin{eqnarray*}
-F_1(H(z))&=& -\sum_{j=0}^\infty c_j \left(\frac{-z}{1-z}\right)^{j+1} = 
\sum_{j=0}^\infty\sum_{n=0}^\infty 
(-1)^j c_j 
\binom{n+j}{n}
z^{n+j+1}
\\
&=& \sum_{k=0}^\infty \sum_{j=0}^k
\binom{k}{j}
 (-1)^j c_j 
z^{k+1}
= \sum_{k=0}^\infty
\hat c_k z^{k+1} = \hat F_1(z)\ .
\end{eqnarray*}
Because $-H$ is a Pick function in $P(-\infty,1)$ which maps $(-\infty,1)$ onto $(-1,\infty)$,
and  $z\mapsto -F_1(-z)$ is a Pick function analytic on  $\C\setminus(-\infty,-1]$, it follows
$\hat F_1$ is in $P(-\infty,1)$. Thus $(\hat c_k)_{k\ge0}$ is completely monotone
by Theorem \ref{t:F}.
 \end{example}

\section{Moments of convex and concave distribution functions}
\label{s:convex}
Work of Diaconis and Freedman \cite{DiaconisFreedman2004}
included a  characterization of
moments of probability distributions admitting monotone densities in terms of 
the triangular array  given by 
\[
c_{n,m} = \binom{n}{m}(I-S)^{n-m}c_m\ , \quad 0\le m<n\,.
\]
Subsequently, Gnedin and Pitman \cite{GnedinPitman2008} established the following criterion 
that characterizes moments of increasing densities in terms of complete monotonicity. 
Such densities correspond to distribution functions $\mu$ that are convex.
We will work with distribution functions that in general satisfy 
$\mu(0)=0\le \mu(0^+)$ and are right continuous on $(0,1]$.
If the distribution function $\mu$ is convex, necessarily $\mu(0^+)=0$
and the measure $\mu$ has no atoms except possibly at $t=1$.

\begin{definition} 
A sequence $a=(a_n)_{n\ge0}$ is {\em completely alternating} if the sequence
given by 
$(S-I)a = (a_{n+1}-a_n)_{n\ge0}$ is completely monotone. 
\end{definition}

\begin{thm}[\cite{GnedinPitman2008}]\label{t:GP}
A sequence $(c_n)_{n\ge0}$ is the sequence of moments of a probability distribution 
on $[0,1]$ having a convex distribution function $\mu$ 
if and only if $c_0 = 1$ and the sequence $(a_n)_{n\ge0}$ defined by
 \begin{equation}\label{d:aseq}
 a_0=0,\qquad a_n = n c_{n-1}, \quad n=1,2,\ldots
 \end{equation}
 is completely alternating.
 \end{thm}

From Theorem \ref{t:F}, we directly obtain the following characterizations of 
completely alternating sequences, and moments of a convex distribution function, 
in terms of the corresponding generating functions. 
We find it convenient to consider distribution functions that are 
not necessarily normalized to be probability distribution functions.
Note that if the sequence $a$ has generating function $A$, 
then $\hat a=(S-I)a$ has upshifted generating function 
\begin{equation}
\hat A_1(z) = (1-z)A(z) - A(0)\ .
\end{equation}
And if $a$ is determined as in Theorem~\ref{t:GP}, from $c$ with upshifted generating
function $F_1$, then $A(z)=zF_1'(z)$.
\begin{prop}
A sequence $(a_n)_{n\ge0}$ is completely alternating if and only if 
its generating function $A$
has the property that the function $(1-z)A(z)$ is 
a Pick function analytic on $(-\infty,1)$.
\end{prop}

\begin{thm} \label{t:convex}
Let $c=(c_n)_{n\ge0}$ be a real sequence with upshifted generating function $F_1$. 
Then the following are equivalent.
\begin{itemize}
\item[(i)] $c$ is the sequence of moments of a convex distribution
function $\mu$ on $[0,1]$ with $\mu(0)=0$. 
\item[(ii)] The sequence $a$ determined from $c$ by \qref{d:aseq} is completely alternating.
\item[(iii)] The function $\hat A_1(z)=(1-z)z F_1'(z)$
is a Pick function analytic on $(-\infty,1)$.
\item[(iv)] The function $\hat A(z) = (1-z) F_1'(z)$ 
is a Pick function analytic and nonnegative on $(-\infty,1)$.
\end{itemize}
\end{thm}

Criteria for $(c_n)_{n\ge0}$ to be the moments of a concave distribution 
function $\mu$, whose corresponding measure has a decreasing density 
on $(0,1]$ with possible atom at 0, turn out to be slightly simpler.

\begin{thm} \label{t:concave}
Let $c=(c_n)_{n\ge0}$ be a real sequence with upshifted generating function $F_1$.
Then the following are equivalent.
\begin{itemize}
\item[(i)] $c$ is the sequence of moments of a concave distribution
function $\mu$ on $[0,1]$.
\item[(ii)] The sequence $((n+1)c_n)_{n\ge0}$ is completely monotone.
\item[(iii)] The function $F_2(z)=z F_1'(z)$
is a Pick function analytic on $(-\infty,1)$.
\item[(iv)] The function $F_1'$ 
is a Pick function analytic and nonnegative on $(-\infty,1)$.
\end{itemize}
\end{thm}
Although this result can be derived from Theorem~\ref{t:convex} by
using reflection as in Example~\ref{ex:reflect} above, we prefer to
illustrate the use of Theorem~\ref{t:F} by providing a self-contained proof. 
Essentially the idea is to consider mixtures of uniform distributions on $[0,s]$,
similar to \cite{GnedinPitman2008}. In particular we make use of the following 
identity valid for $0<s\le 1$ and $z\in \C\setminus[1,\infty)$:
\begin{equation}\label{e:sz}
\frac1{1-sz} 
= \frac d{dz} \int_0^1 \frac z{1-tz} \left(\frac1s \one_{0<t<s}\right)\,dt.
\end{equation}
\begin{proof} Since (ii), (iii) and (iv) are equivalent by Theorem~\ref{t:F}, it
remains to prove (iv) is equivalent to (i). Assume (iv).
Then by Theorem~\ref{t:F},
there is a finite positive measure $\nu$ on $[0,1]$ 
with $\nu\{0\}=0$, and $a_0\ge0$, such that 
\begin{equation}\label{e:F1pr}
F_1'(z) = a_0+ \int_0^1 \frac 1{1-sz}\,d\nu(s) 
\end{equation} 
(Here we have separated out the mass $a_0$ of any atom at $s=0$.)
Using \qref{e:sz}, integration, division by $z$, and Fubini's theorem yields
\begin{equation}
F(z) -a_0 = 
\int_0^1  
\left( \int_0^1 \frac 1{1-tz} \left(\frac1s \one_{0<t<s}\right)dt \right)d\nu(s)
= \int_0^1 \frac 1{1-tz} w(t)\,dt\,,
\label{e:Fw}
\end{equation}
where 
\begin{equation}\label{e:wdef}
w(t) = \int_0^1 \left(\frac1s \one_{0<t<s}\right)\,d\nu(s) = \int_t^1 \frac{d\nu(s)}s\,.
\end{equation}
This function $w$ is decreasing on $(0,1]$, and taking $z=0$
in \qref{e:Fw} shows that $w$ is integrable on $(0,1)$, with
\[
F(0)= a_0+ \int_0^1 w(t)\,dt  \,. 
\]
Then the distribution function $\mu(t)=a_0+\int_0^t w(r)\,dr$, $t>0$, is concave, and 
\qref{e:F0def} holds as desired, proving (i).

For the converse, assume (i). Then \qref{e:Fw} holds for some $a_0\ge0$ and $w$ 
decreasing and integrable on $(0,1]$. Extend $w$ by zero for $t>1$ and define 
the measure $\nu$ such that $d\nu(s)=-s\,dw(s)$ on $(0,1]$. 
Then \qref{e:wdef} holds, and  $\int_0^1w(t)\,dt=\int_0^1d\nu(s)$
by Fubini's theorem, and one deduces \qref{e:F1pr} by reversing the steps above.
Thus (iv) follows. 
\end{proof}

Finally, we derive Theorem~\ref{t:convex} from Theorem~\ref{t:concave},
using reflection as in Example~\ref{ex:reflect}.
\begin{proof}[Proof of Theorem~\ref{t:convex}]
Suppose (i) $(c_n)_{n\ge0}$ has the moment representation \qref{e:cjcm} with convex distribution
function $\mu$, corresponding to an increasing density (with a possible atom at $1$).
Then by Example~\ref{ex:reflect}, the sequence of leading differences
$(\hat c_k)_{k\ge0}$ is represented by the reflected density (with possible atom at $0$).
This density is decreasing, implying the leading differences are moments of a 
concave distribution function.
The upshifted generating function $\hat F_1=-F_1\circ H$ given by
\qref{e:hatF1} therefore has the property that 
\[
\hat F_2(z)=z\hat F_1'(z) 
= \frac{z}{(1-z)^2} F_1'(H(z)) \,,
\]
and this is a Pick function analytic on $(-\infty,1)$. But since $H(H(z))=z$, 
\[
-\hat F_2(H(z)) = -H(z)(1-z)^2 F_1'(z) = z(1-z) F_1'(z) = \hat A_1(z)
\]
and this is a Pick function analytic on $(-\infty,1)$. This yields (iii), and (ii) and (iv)
are equivalent by Theorem~\ref{t:F}. 

In the converse direction, if we assume the sequence $a$ derived from $c$ 
in \qref{d:aseq} is completely alternating,
then $\hat A_1$ is Pick. After reversing the arguments we deduce that the leading
difference sequence $\hat c=(\hat c_k)_{k\ge0}$ is represented by a decreasing density. 
By Example~\ref{ex:reflect}, $c$ itself is represented by the reflected, increasing density,
hence by a convex distribution function.
\end{proof}

\section{Fuss-Catalan and binomial sequences}

For any real $p$ and $r$, the general \textit{Fuss-Catalan} numbers 
\cite{Graham, Mlotkowski},
also called \textit{Raney} numbers \cite{PensonZ,MlotkowskiPZ},
are defined by $A_0(p,r)=1$ and
\begin{equation}\label{d:FC}
A_n(p,r) = \frac{r}{n!} \prod_{j=1}^{n-1} (pn+r-j)
= \frac{r}{pn+r} \binom{pn+r}{n},
\quad n=1,2,\ldots
\end{equation}
For $r=1$ one has 
\[
A_n(p,1) = \frac{(pn)!}{n! (pn+1-n)!} = \frac{1}{pn-n+1} \binom{pn}{n}\ , 
\]
showing that in particular, $A_n(2,1)$ is the $n$th Catalan number
$C_n=\frac1{n+1}\binom{2n}{n}$.
The Fuss-Catalan numbers have a long history and a very large number of 
fascinating interpretations and applications, of which we only mention a couple.  
E.g., $A_n(m,1)$ counts the number of $m$-ary trees with $n$ nodes,
and $A_n(m,k)$ counts the number of sequences of 
$mn+k$ terms selected from $\{1,1-m\}$, which sum to $1$ while having 
partial sums that are all positive \cite[eq.~(7.70)]{Graham}. 

\subsection{Fuss-Catalan moments}
Recently, Alexeev \etal~\cite{Alexeev} proved that $A_n(m,1)$ arises in random matrix theory,
as the $n$th moment of the limiting distribution of scaled squared singular values of products of $m$ 
complex matrices with random i.i.d.\ entries having zero mean, unit variance, and bounded fourth moments. 
Here, we make use of Corollary~\ref{c:ab} to provide a simple analytic proof 
of the following characterization of those real $p$ and $r$ with $p>0$
for which \qref{d:FC} yields
a dilated Hausdorff moment sequence.

\begin{thm}\label{t:FC}
Let $p$ and $r$ be real with $p>0$. Then  $(A_n(p,r))_{n\ge0}$ 
is the sequence of moments of some probability distribution $\mu_{p,r}$ 
having compact support in $[0,\infty)$ if and only if  $p\ge 1$ and $p\ge r\ge0$. 
In this case, $\mu_{p,r}$ is supported in the minimal interval $[0,\tau_p]$ with
 $\tau_p = p^p/(p-1)^{p-1}$ for $p>1$, $\tau_1=1$.
\end{thm}
Mlotkowski \cite{Mlotkowski} proved the \textit{if} part of this theorem
using techniques from free probability theory, and the \textit{only if} part
was subsequently established by Mlotkowski and Penson~\cite{MlotkowskiP14}
using arguments that involve monotonic convolution. 
Explicit representations of densities $W_{p,r}(t)$ for $\mu_{p,r}$ 
and more general probability distributions 
have been provided by Penson and Zyczkowski \cite{PensonZ} and 
Mlotkowski \etal~\cite{Mlotkowski} 
in terms of the Meijer $G$ function and hypergeometric functions. 
Haagerup and M\"oller~\cite{HaagerupMoller} have derived parametrized 
forms of the densities $W_{p,1}(t)$ in terms of trigonometric functions.
These distributions provide generalizations of previously known distributions such
as the Marchenko-Pastur distribution
\[
d\mu_{2,1}(t) = \frac1{2\pi} \sqrt{\frac{4-t}t}\,dt,
\]
and $d\mu_{2,2}(t)=t\,d\mu_{2,1}(t)$,
given by the Wigner semi-circle law centered on $t=2$.

Due to Corollary~\ref{c:ab}, Theorem~\ref{t:FC} is directly implied by the following.

\begin{lem}
Let $p$ and $r$ be real with $p>0$, and let $B_{p,r}$ denote the generating function for the sequence $(A_n(p,r))_{n\ge0}$.
Then $B_{p,r}$ is a Pick function analytic and nonnegative on some interval $(-\infty,1/\tau)$, $0<\tau<\infty$,
if and only if $p\ge 1$ and $p\ge r\ge0$. If this is the case, the minimal $\tau$ is $\tau_p$.
\end{lem}

\begin{proof} The proof breaks into various cases.\footnote{See the Corrigendum below for repairs to gaps in steps 2 and 4.}

1. \textit{The case $p=1=r$.} Since $A_n(1,1)=1$ for all $n$, $B_{1,1}(z)=1/(1-z)$
and this is a Pick function analytic and positive on the maximal interval $(-\infty,1)$.

2. \textit{The case $p>1=r$.}
Let $B_p=B_{p,1}$ denote the generating function for $(A_n(p,1))_{n\ge0}$. 
It is well known \cite{Graham, Mlotkowski}
that $B_p(z)$ satisfies the functional relation
\begin{equation}\label{eq:Cz}
B_p(z) = 1+ z B_p(z)^p \ ,
\end{equation}
as can be checked using the Lagrange inversion formula. 
Note $B_p(z)$ cannot vanish, so \qref{eq:Cz} is equivalent to  
\begin{equation}\label{e:Cpz}
z = \psi_p( B_p(z)), \qquad \psi_p(c) = \frac{c-1}{c^p} = c^{1-p}-c^{-p},
\end{equation}
The function $\psi_p$ is analytic and strictly increasing on the 
interval $(0,p/(p-1))$, rising from $-\infty$ to the value 
\[
  z_p : =1/\tau_p = (p-1)^{p-1}/p^p \ .
\] 
By consequence $B_p$ is analytic, positive,
and strictly increasing on $(-\infty,1)$, 
hence satisfies the Pick property in a neighborhood of this interval.

We continue $B_p$ analytically to the domain $\C\setminus[z_p,\infty)$
by using a differential equation that $B_p$ satisfies. Namely, \qref{e:Cpz} implies
\begin{equation}\label{e:Code}
B_p'(z) = \frac1z\, \frac{B_p(B_p-1)}{p-(p-1)B_p} \ .
\end{equation}
We integrate along rays $t\mapsto  t e^{i\theta}$ from $t=t_0$ near the origin, for fixed
$\theta\in(0,2\pi)$.
By continuation theory for ordinary differential equations,
the solution exists for $t$ in some maximal interval $[t_0,T)$
with the property that if $T<\infty$, then as $t\uparrow T$, either $|B_p|\to\infty$ 
or $B_p\to p/(p-1) $. The first case is not possible since by \qref{e:Cpz}, 
$|B_p|\to\infty$ implies $te^{i\theta}\to0$.
And the second case is not possible since by \qref{e:Cpz}, $B_p\to p/(p-1)$ 
implies $t e^{i\theta}\to z_p$.  Therefore $T=\infty$ for every $\theta$.

This provides an analytic continuation of $B_p$ to $\C\setminus[z_*,\infty)$. 
Necessarily $\im B_p>0$ everywhere
in the upper half plane $\im z>0$, since $\im B_p$ cannot vanish 
due to \qref{e:Cpz}. Hence $B_p$ is a Pick function,
and is analytic and positive on the maximal interval $(-\infty, z_p)$.
This finishes the proof for $p>r=1$.

3. \textit{The case $p\ge1$, $p\ge r\ge0$.} 
For any real $p$ and $r$, the generating function $B_{p,r}$ 
satisfies the Lambert equation
\begin{equation}\label{e:Bpr}
B_{p,r}(z) = B_p(z)^r \,.
\end{equation}
(See \cite[eq.~(5.60)]{Graham} and \cite[eq.~(3.2)]{Mlotkowski},
which is based on \cite[p.~148]{Riordan}.)
Since $B_p$ is analytic and never vanishes or takes negative values,
$\Log \circ B_p$ is a Pick function, and 
$B_p(z)^r=\exp(r\Log B_p(z))$.
Thus $B_p(z)^r$ is analytic in the upper half plane, 
and positive and increasing on the maximal interval 
$(-\infty,z_p)$, with limit $0$ at $-\infty$ and value $1$ at $z=0$.

We claim $B_{p,r}$ is a Pick function:
Note $zB_p^p(z)=B_p(z)-1$ and this is a nontrivial Pick function analytic
on $(-\infty,z_p)$ that vanishes at $z=0$. 
By Corollary~\ref{c:ab}, then, $B_p(z)^p$ itself must be a Pick function.
Since $z\mapsto z^{r/p}$ is a Pick function, it follows  $(B_p(z)^p)^{r/p}=B_p(z)^r$ is a Pick function.

4. \textit{The case $r>p\ge1$.} Recall $zB_p(z)^p$ is a Pick function
analytic on $(-\infty,1)$. This function is negative for $z<0$,
so necessarily  
\[
\arg zB_p(z)^p=\pi\,,\quad z<0.
\]
Since $0<\arg z<\pi$ implies $0<\arg B_p(z)<\pi$, it follows that by taking
$z=e^{i\theta}$, the quantity 
\[
\arg zB_p(z)^r =  \arg z B_p(z)^p + (r-p)\arg B_p(z)
\]
takes values ranging from $0$ to more than $\pi$ as $\theta$ varies from 0 to $\pi$.
Consequently $zB_p(z)^r$ cannot be a Pick function.

5. \textit{The case $0<p<1$, $r>0$}. In this case, $\psi_p$ is globally strictly monotone
on  $(0,\infty)$ and maps this interval analytically onto $\R$. 
This means the inverse function $B_p$ is globally real analytic on $\R$,
and the same is true for $B_{p,r}$ for any $r>0$ by \qref{e:Bpr}. 
In this case, $B_{p,r}$ cannot be a Pick function. 
Indeed, the only Pick functions analytic 
and positive globally on $\R$ are constant. This is because
they have a general representation as in \qref{e:Grep}, 
with $\mu$
as in \qref{e:mudef}. Hence $d\mu=0$
for a Pick function globally analytic on $\R$.

6. \textit{The case $p>0>r$.} Due to the Lambert equation \qref{e:Bpr},
$B_{p,r}$ is decreasing on $(-\infty,z_p)$ in this case,
hence $B_{p,r}$ cannot be a Pick function.
\end{proof}

We next apply Theorem~\ref{t:concave} to deduce
that the distribution functions $\mu_{p,1}$ have decreasing densities for $p\ge2$.  
Numerical plots in \cite{MlotkowskiPZ} indicate that this condition 
is not sharp for noninteger values of $p$. 
See \cite{PensonZ,MlotkowskiPZ} for detailed information concerning 
densities $W_{p,r}$ for all $p\ge1$, $0<r\le p$.
Note, however, that for $r=p$ one has 
$A_n(p,p)=A_{n+1}(p,1)$, whence follows the simple relation
\begin{equation}\label{e:mupp}
d\mu_{p,p}(t) = t\, d\mu_{p,1}(t)\,.
\end{equation}

\begin{prop} If $p\ge2$, then the probability distribution function
$\mu_{p,1}$ is concave, and continuous at 0. We have
\[
d\mu_{p,1}(t) = W_{p,1}(t)\,dt, \qquad d\mu_{p,p}(t) = t \,W_{p,1}(t)\,dt,
\]
where $W_{p,1}$ is a decreasing, integrable function on $(0,1]$.
\end{prop}

\begin{proof} The function $F_1=zB_{p}(z)$ satisfies 
\begin{equation}\label{e:F1p}
F_1' = \frac{p-2}{p-1}B_p + \frac1{p-1} \left(\frac{p}{p-(p-1)B_p}-1\right)\ ,
\end{equation}
and if $p\ge2$, this is a Pick function analytic and positive on $(-\infty,1)$,
fulfilling condition (iv) of Theorem~\ref{t:concave}.
By Remark~\ref{r:mass}, the measure $\mu_{p,1}$ has no atom at 0,
since $B_p(z)\to0$ as $z\to-\infty$.
\end{proof}

\subsection{Fuss-Catalan canonical sequences and densities}
Next we study the {canonical sequences and densities}
that are associated with Fuss-Catalan sequences
according to Theorem~\ref{t:FC} and Proposition~\ref{p:rho}.
\begin{thm}\label{t:wp}
For every $p\ge1$, there is a canonical sequence $(b_n^{(p)})$, given explicitly by
\begin{equation}
\label{e:bpfc}
b_{n-1}^{(p)} = n \sum_{k=1}^n \frac{(-1)^{k-1}}k \binom{n}{k} A_n(p,k) 
= \binom{pn-1}{n-1}\ ,
\quad n=1,2,\ldots,
\end{equation}
 and a nonincreasing density 
$w_p\colon[0,\tau_p]\to [0,1/p]$ satisfying $pw_p(0^+)=1$ and $w_p(\tau_p^-)=0$, 
such that
\begin{equation}\label{e:logBp}
\frac1r \Log B_{p,r}(z) =  \sum_{n=0}^\infty \frac{b_n^{(p)}}{n+1}z^{n+1} = 
  \int_0^{\tau_p} \frac z{1-tz}\,w_p(t)\,dt\,,\ \ 0<r\le p.
\end{equation}
Moreover, $(b_n^{(p)}\tau_p^{-n})_{n\ge0}$ is a completely monotone sequence.
\end{thm}
\begin{proof} For $r>0$, the sequence $a\supr=(A_n(p,r))_{n\ge0}$
is a dilated Hausdorff sequence if and only if $r\le p$.
By Proposition~\ref{p:rho}, we infer \qref{e:logBp},
with $p\esssup w_p(t) =1$. (Note: $w_1(t)\equiv1$.)
To prove that $w_p$ is nonincreasing, observe that \qref{e:Code} implies
\begin{equation}\label{e:Bpder}
(p-1) \frac{z B_p'(z)}{B_p(z)} = \frac1{p-(p-1)B_p(z)} -1,
\end{equation}
and this is a Pick function analytic on $(-\infty,1/\tau_p)$. 
Let 
\begin{equation}\label{e:tG12}
\tilde G_1(z) = \Log B_p(z/\tau_p) = \int_0^1 \frac z{1-sz}\, w_p(\tau_p s)\,ds.
\end{equation}
Then $z\tilde G_1'(z)=\hat z B_p'(\hat z)/B_p(\hat z)$ ($\hat z=z/\tau_p$)
is a Pick function analytic on $(-\infty,1)$.  Theorem~\ref{t:concave}
implies the density $s\mapsto w_p(\tau_p s)$ is nonincreasing on $(0,1]$. 
Therefore $pw_p(0^+)=1$. (See \qref{e:wpzero} below for the proof that $w_p(\tau_p^-)=0$.)

The first expression in \qref{e:bpfc} follows from \qref{e:bnformula}.
Noticing that \qref{e:Bpder} implies
\begin{equation}\label{e:bppbp}
\sum_{n=0}^\infty b^{(p)}_n z^n =\frac{B_p'(z)}{B_p(z)}= \frac{B_p(z)^{p-1}}{1-p+pB_p(z)^{-1}}\ ,
\end{equation}
the second expression follows from identity (5.61) in \cite{Graham} (setting $r=p-1$).
Finally, plugging $\hat z=z/\tau_p$ into the power series in \qref{e:logBp}
and using Theorem~\ref{t:concave} part (ii), we infer
$(b_n^{(p)}\tau_p^{-n})$ is completely monotone as claimed.
\end{proof}

\subsection{Binomial sequences}
In a recent paper, Mlotkowski and Penson~\cite{MlotkowskiP14} established necessary
and sufficient conditions for the binomial sequence 
\begin{equation}\label{e:biseq}
\binom{pn+r-1}{n} \qquad n=0,1,\ldots
\end{equation}
to be the moment sequence of some probability distribution on some interval $[0,\tau]$. 
In this subsection we will provide an alternative proof of this characterization
based on Corollary~\ref{c:ab}.  First, however, we note that the case $r=1$ is connected
with the canonical density $w_p$ described in Theorem~\ref{t:wp}.

\begin{cor}
For every real $p>1$, the binomial sequence 
$\binom{pn}{n}$, $n=0,1,\ldots$ is the moment sequence for
the probability distribution function $1-pw_p(t)$ on $[0,\tau_p]$, having
 generating function
\begin{equation}\label{e:biwp}
\sum_{n=0}^\infty \binom{pn}n z^n = \int_0^{\tau_p} \frac1{1-tz} \,d(1-pw_p(t)).
\end{equation}
\end{cor}

\begin{proof}
Noting that $\binom{pn}n = pb_{n-1}^{(p)}$ for $n\ge1$, we use \qref{e:logBp}
to compute that
\begin{eqnarray*}
\sum_{n=0}^\infty \binom{pn}n z^n &=& 1+ pz 
\int_0^{\tau_p} \frac{\D}{\D z} \left(\frac z{1-tz}\right)\,w_p(t)\,dt \\
&=& 
1+ \int_0^{\tau_p} \frac{\D}{\D t} \left(\frac 1{1-tz}\right)\,pw_p(t)\,dt \\
&=& 
\int_0^{\tau_p}\frac 1{1-tz} \, d(1-pw_p(t)) + \frac{pw_p(\tau_p^-)}{1-\tau_pz}
\ .
\end{eqnarray*}
Comparing the last line of this calculation to the first and using \qref{e:logBp},
\qref{e:Bpder}, and \qref{e:Cpz}, 
we find
\begin{eqnarray}
w_p(\tau_p^-) &=& \lim_{z\uparrow 1/\tau_p} (1-\tau_p z) z B_p'(z)/B_p(z) 
\nonumber \\&=&
\frac{1}{p-1}  \lim_{z\uparrow 1/\tau_p} \frac{1-\tau_p z}{p-(p-1)B_p(z)}
\nonumber \\&=&
\frac{\tau_p}{(p-1)^2} \psi_p'\left(\frac p{p-1}\right) = 0.
\label{e:wpzero} \end{eqnarray}
This finishes the proof.
\end{proof}

In the case $p=2$ one can obtain an explicit formula for $w_2(t)$
by elementary means. From the formula
\begin{equation}
\binom{2n}{n} = \frac1\pi \int_0^\pi (4\cos^2(u/2))^n\,du \ ,
\end{equation}
we find
\begin{equation}
\sum_{n=0}^\infty \binom{2n}n z^n = \frac1\pi \int_0^\pi \frac{1}{1-4z\cos^2(u/2)}\,du\,.
\end{equation}
By comparing with \qref{e:biwp}, we deduce that 
\begin{equation}
w_2(t) = \frac1\pi \arccos \sqrt{\frac t4}.
\end{equation}

\begin{remark} For an arbitrary $p>1$, an explicit formula for the inverse of $w_p$
may be obtained in a similar way, based upon the integral representation
formula \qref{e:rk} for binomial coefficients. 
Set $k=n$, $r=pn$ and 
\[
f_p(u) =  \frac{\sin^p \pi u}{\sin (\pi u/p) \sin^{p-1}((1-1/p)\pi u)}\ .
\]
Then $f_p(u)$ decreases from $\tau_p$ to 0 as $u$ increases from $0$ to $1$, and
we deduce that
\begin{equation}
\sum_{n=0}^\infty \binom{pn}n z^n = 
 \int_0^1\frac1{1-f_p(u)z}\,du 
= \int_0^{\tau_p} \frac1{1-tz}d(1-f_p\inv(t))\ .
\end{equation}
By comparing with \qref{e:biwp}, it follows
\begin{equation}
p w_p(t) = f_p\inv(t), \quad 0<t<\tau_p \ .
\end{equation} 
\end{remark}

\begin{thm}[\cite{MlotkowskiP14}]
Let $p$ and $r$ be real with $p>0$. The sequence \qref{e:biseq}
is the sequence of moments of some probability distribution $\nu_{p,r}$ 
having compact support in $[0,\infty)$ if and only if  $p\ge 1$ and $p\ge r\ge0$. 
In this case, $\nu_{p,r}$ is supported in the minimal interval $[0,\tau_p]$ with
 $\tau_p = p^p/(p-1)^{p-1}$ for $p>1$, $\tau_1=1$.
\end{thm}

\begin{proof} The generating function $E_{p,r}(z)$ for the sequence \qref{e:biseq}
is known to satisfy (see \cite{MlotkowskiP14} and \cite[eq.~(5.61)]{Graham})
\begin{equation}
E_{p,r}(z) = \frac{B_p(z)^{r}}{p-(p-1)B_p(z)} \,.
\end{equation}
From what was proved before, $E_{p,r}$ (like $B_p$) is analytic and nonnegative
on $(-\infty,z_p)$, and analytic in the upper half plane.
By Corollary~\ref{c:ab}, it suffices to show that $E_{p,r}$ is a Pick function 
if and only if $p\ge1$ and $p\ge r\ge0$.
Similarly to the proof of Theorem~\ref{t:FC}, the proof breaks into several cases.

1. \textit{The case $r<0$.} In this case, $E_{p,r}(z)$ is decreasing in $z$ for large $z<0$, 
hence $E_{p,r}$ cannot be a Pick function. 

2. \textit{The case $0<p<1$, $r\ge0$.}  As in case 5 of the proof of Theorem~\ref{t:FC},
$B_p$ and hence $E_{p,r}$ is globally analytic and positive on $\R$. The only Pick functions
with this property are constant, so $E_{p,r}$ is not a Pick function.

3. \textit{The case $p>1$, $0\le r\le p$.}  By \qref{e:bppbp}, $(p-1)zE_{p,p}(z)$
equals the right-hand side of \qref{e:Bpder} and therefore is a Pick function. 
Now, for $0<\arg z<\pi$ we have, on the one hand, that
\begin{equation}
\arg z E_{p,r}(z) = \arg z+ \arg E_{p,0}(z) + r\arg B_p(z) >0,
\end{equation}
and on the other hand, that
\begin{equation}\label{e:zDpr}
\arg zE_{p,r}(z) = \arg zE_{p,p}(z)+ (r-p) \arg B_p(z) <\pi\ .
\end{equation}
This shows $zE_{p,r}(z)$ is a Pick function, hence $E_{p,r}$ is a Pick function by
Corollary~\ref{c:ab}.

4. \textit{The case $p>1$, $r>p$.}  In this case we claim $zE_{p,r}(z)$ is not a Pick function.
To see this, take $z=e^{i\theta}$ for $0\le \theta\le \pi$ and note that
in \qref{e:zDpr},
the last term is positive and the first term varies from $0$ to (at least) $\pi$.
Hence the sum is somewhere more than $\pi$ for some $z$ in the upper half plane.

5. \textit{The case $p=1$.} In this case, $E_{1,r}(z)=(1-z)^{-r}$, and this is a Pick function
if and only if $0\le r\le 1$.
\end{proof}

\begin{remark}
 For rational $p\ge1$, with $p\ge1+r>0$, explicit formulae
in terms of the Meijer $G$ function have been derived 
by Mlotkowski and Penson \cite{MlotkowskiP14}
for a density denoted $V_{p,r}(t)$ with the property that
\begin{equation}
\sum_{n=0}^\infty \binom{pn+r}n z^n = \int_0^{\tau_p} \frac1{1-tz} V_{p,r}(t)\,dt\,.
\end{equation}
By comparing with the above, it follows 
\begin{equation}
-pw_p'(t) = V_{p,0}(t).
\end{equation}
\end{remark}

\section*{Acknowledgements}
We thank Tewodros Amdeberhan (private communication) for finding
the second expression in  \qref{e:bpfc} from the first 
by using the Zeilberger algorithm~\cite{AeqB}.
The authors are grateful to Pierre Degond for initiating 
a collaboration which led to this work. 
RLP is grateful to Govind Menon for discussions on Pick functions.
This material is based upon work supported by the National
Science Foundation under 
grants DMS 1211161 and RNMS11-07444 (KI-Net) 
and partially supported by the Center for Nonlinear Analysis (CNA)
under National Science Foundation grant 0635983.

We are also grateful to Alan Sokal for finding the gaps in the
proof of Lemma~3 and suggesting improved proofs.


\bibliographystyle{plain}
\bibliography{discretecm-arxiv.bbl}



\appendix

\vfil\pagebreak
\section*{Corrigendum}

There are gaps in Steps 2 and 4 of the proof of Lemma~3, 
and an inaccuracy in equation (9).
Lemma~3 directly implies Theorem~5, which states in part that
for any real $p$ and $r$, the sequence $a_{p,r}=(A_n(p,r))_{n\ge0}$ 
of Fuss-Catalan (or Raney) numbers
is the sequence of moments of a probability distribution having compact support in $[0,\infty)$
if and only if $p\ge1$ and $p\ge r\ge0$.
The generating functions $B_{p,r}$ associated to the sequences $a_{p,r}$ are known to satisfy
the functional equations
\begin{equation}\label{e:f1}
B_{p,1}(z) = 1+ z B_{p,1}(z)^p\,,
\end{equation}
\begin{equation}\label{e:f2}
B_{p,r}(z) = B_{p,1}(z)^r \,.
\end{equation}

Step 2 of the proof of Lemma~3 asserts that if $p>1=r$, then $B_p = B_{p,1}$ is a Pick function
analytic and nonnegative on a certain interval $(-\infty,z_p)$. There is a gap in the argument,
however, due to a failure to globally control the argument of $B_p(z)^p$, 
so that if $B_p(z)$ approaches a real number, $B_p(z)^p$ need not.

The analytic continuation of $B_p(z)$ to $\C\setminus[z_p,\infty)$ is carried out
by integrating a differential equation implied by the functional equation \eqref{e:f1}, namely
\begin{equation}\label{e:bpode}
B_p'(z) = \frac1z\frac{B_p(B_p-1)}{p-(p-1)B_p} \,,
\end{equation}
along rays $t\mapsto te^{i\theta}$ with $\theta\in(0,\pi)\cup(-\pi,0)$. 
We restrict attention to $\theta\in(0,\pi)$, as the case $\theta\in(-\pi,0)$
is similar with some signs changed. 

To be more precise, writing $w(t,\theta)=B_p(te^{i\theta})$ we solve the differential equation
\begin{equation}\label{e:wode}
\frac{\D w}{\D t} = \frac1t \frac{w(w-1)}{w-p(w-1)}, \qquad w(t_0,\theta) = w_0 \,,
\end{equation}
taking initial value $w_0=B_p(t_0e^{i\theta})$ as given for some small $t_0>0$ by the series.
Since $B_p'(0)=A_1(p,1)=1$, this initial value has positive imaginary part if $t_0$ is small enough. 
By continuation theory for ordinary differential equations, 
a unique solution exists for $t$ in a maximal interval $[t_0,T)$, with the property that  
if $T<\infty$ then as $t\uparrow T$,  $w(t,\theta)$ must exit any compact subdomain 
of $\C\setminus\{q\}$ where $q=\frac{p}{p-1}$.  That is, 
for $t\uparrow T<\infty$, either $|w(t,\theta)|\to\infty$ or $w(t,\theta)\to q$.

We now argue differently than in the proof of Lemma~3 above, 
noting that the real line $\R$ is an invariant
set for \eqref{e:wode}. That is, if $w(t_1,\theta)$ is ever real for some $t_1\in[t_0,T)$, 
then $w(t,\theta)$ is real for all $t$. 
Since $w_0$ is not real, necessarily $w=w(t,\theta)$ has positive imaginary part
for all $t\in[t_0,T)$, and the same for $w-1$. By separating variables in \eqref{e:wode} we find
\begin{equation}\label{e:logw}
\Log(w-1) - p\Log w = \log t + i\theta \,,
\end{equation}
as this holds when $t=t_0$ due to \eqref{e:f1}. It follows
$\arg w\in(0,\frac{\pi-\theta}p)$ since  
\begin{equation}\label{e:argw}
0< p \arg w = \arg({w-1}) - \theta <\pi - \theta. 
\end{equation}

The rest of the proof in Step 2 goes as in the given proof of Lemma 3, 
using the functional equation:
If $|w|\to\infty$ as $t\uparrow T$ then by \eqref{e:logw},
$te^{i\theta}= (w-1)/w^p\to0$, contradiction, 
while if $w\to q$ then $te^{i\theta}\to (q-1)/q^p>0$, also a contradiction.
Hence $T=\infty$.
Thus we obtain an analytic continuation of $B_p(z)$ to the upper half plane 
having everywhere positive imaginary part there, and this completes the 
proof in Step 2.

\medskip
The proof in Step 3 is simpler now, since whenever $\im z>0$, from \eqref{e:argw} 
we get
\[
0<\arg B_p(z) <(\pi-\arg z)/p 
\]
for $p>1$, and also for $p=1$ because $B_1(z)=1/(1-z)$.
Then for $p\ge1$ and $p\ge r\ge0$ we have $0<\arg B_p(z)^r<\pi$ 
in the upper half plane.
Hence $B_p(z)^r = \exp(r\Log B_p(z))$ is a Pick function 
and it is analytic and positive on $(-\infty,z_p)$.

\medskip
In Step 4, we want to show that if $r>p\ge1$ then $B_p(z)^r$ is not a Pick function. 
With $w=B_p(te^{i\theta})$ as above, from \eqref{e:f1} it follows that
$w\to0$ as $t\to\infty$, whence from \eqref{e:argw} (valid also for $p=1$)
we infer $\arg w\to (\pi-\theta)/p$.  If we suppose $B_p(z)^r$ is Pick, then as 
$t\to\infty$ we find
$ \arg w^r =  r \arg w \to \frac rp (\pi-\theta).$
This is greater than $\pi$ if $\theta>0$ is small enough, yielding a contradiction. 
Hence $B_p(z)^r$ cannot be Pick. 

\medskip
Finally, eq.~(9) 
should read as follows, since $\mu$ is taken as right continuous:
\begin{equation}
\tag{9}
\mu(b)-\mu(a) = \mu(a,b] = \frac{\mu\{b\}-\mu\{a\}}2+
\lim_{h\to0^+} \frac1\pi\int_a^b \im F_*(t+ih) \,dt, \quad a,b\in\R.
\end{equation}

\end{document}